\documentclass[a4paper,11pt,fleqn]{amsart}
\usepackage[utf8]{inputenc}
\usepackage{epsfig,graphics}
\usepackage[all,knot,poly,color,curve]{xy}
\usepackage{array}
\usepackage{amsthm,amssymb,amsbsy,bbm,a4wide,verbatim,url}
\usepackage[marginpar]{todo}
\usepackage{longtable}
\usepackage{rotating}
\usepackage{color}

\usepackage{fancybox}
\usepackage{enumitem}

\DeclareMathOperator{\rk}{\mathrm{rk}}

\def\tr{\mathrm{tr}}

\def\onto{\twoheadrightarrow}
\def\diag{\mathrm{diag}}
\def\om{\omega}

\def\kk{\mathbf{k}}

\def\C{\ensuremath{\mathbbm{C}}}
\def\Q{\mathbbm{Q}}
\def\Qb{\overline{\mathbbm{Q}}}
\def\Z{\mathbbm{Z}}

\def\sl{\mathfrak{sl}}

\newcommand{\la}{\lambda}

\def\ii{\mathrm{i}}
\def\dd{\mathrm{d}}

\def\Ker{\mathrm{Ker}}

\def\End{\mathrm{End}}

\def\Res{\mathrm{Res}}
\def\Ind{\mathrm{Ind}}
\def\Irr{\mathrm{Irr}}

\def\GL{\mathrm{GL}}

\newcommand\bkappa{{\boldsymbol{\kappa}}}

\makeatletter
\def\v@rt#1#2{\m@th\ooalign{$\hfil#1|\hfil$\crcr$#1#2$}}
\def\captr{\mathrel{\mathpalette\v@rt\cap}}
\makeatother

\newtheorem{theorem}{Theorem}[section]
\newtheorem{lemma}[theorem]{Lemma}
\newtheorem{question}[theorem]{Question}
\newtheorem{remark}[theorem]{Remark}

\newtheorem{corollary}[theorem]{Corollary}

\newtheorem{proposition}[theorem]{Proposition}

\numberwithin{equation}{section}

\date{July 28, 2019}

\author{Ivan Marin}

\title[]{Truncations and extensions of the Brauer-Chen algebra} 
\dedicatory{To the memory of Kay Magaard}

\begin{document}

\maketitle

{\bf Abstract.} The Brauer-Chen algebra is a generalization of the algebra of Brauer diagrams to arbitrary complex reflection groups, that admits a natural monodromic deformation. We determine the generic representation theory of the first non trivial quotient of this algebra. We also define natural extensions of this algebra and prove that they similarly admit natural monodromic deformations.

\tableofcontents

\section{Introduction}

\subsection{Context}
The algebra of Brauer diagrams was introduced by Brauer in \cite{BRAUER} in order to understand
the commutant of the orthogonal (or symplectic) groups of the $n$-fold tensor powers of a quadratic space $V$, in the same way as the (group algebra of the) symmetric group $\mathfrak{S}_n$ captures the commutant of the general linear group on the $n$-fold tensor powers of a vector space $V$.
Its structure has been determined by Wenzl in \cite{WENZL}. Combinatorially, it can be described as an extension of the group algebra of $\mathfrak{S}_n$.

It has been shown that natural generalizations of the Brauer algebra exists for other finite reflection groups. In \cite{CFW}, Cohen Frenk and Wales attached one to any Coxeter group of type ADE, the case of type $A_n$ corresponding to the
original case. A complete structural description was obtained in this case.

Later, Chen proposed in \cite{CHENBR} a much larger generalization to \emph{arbitrary} (complex) reflection groups. This generalization is isomorphic to the Cohen-Frenk-Wales algebra for real reflection groups of type ADE.

By contrast with the ADE case, the structural description of this algebra in general is still open, and even the determination of its dimension is still an open problem in general. In this paper we provide a partial description, and more precisely we decompose its first non-trivial quotient ('truncation').

Another important aspect of the Brauer-Chen algebra is that it admits
a natural deformation, obtained by the monodromy of some 1-form, in
the same vein as the deformation of the group algebra known as the Hecke algebra.
It has been proved recently that these Hecke algebras admit natural extensions. We then prove that the deformation of the Brauer-Chen algebra also admits such
a natural extension, also provided by a monodromy construction.

Finally, we will explain how these two aspects should play a role in the construction and understanding of a generalized Krammer representation for
the corresponding braid groups.

\subsection{Truncations}
We introduce a natural series of two-sided ideals $I_1 \supset I_2 \supset \dots$ and call representation of level $r$ of the Brauer-Chen  algebra $Br(W)$ any representation factorizing through $Br_r(W)=Br(W)/I_{r+1}$ but not through $Br_{r-1}(W)$. This series of ideals satisfies that $Br_0(W)$ coincides with the group algebra of $W$, so $Br_1(W)$ is the first non-trivial quotient, and the first new representations are the `level 1 representations' arising there.

Our first main result is then the following one.
\begin{theorem} \label{thm:structBr1}
Let $\bkappa$ be a field of characteristic $0$, $\kk = \bkappa(\delta)$ the function field in one indeterminate. Let $W$ be (pseudo-)reflection group, with set of reflection hyperplanes $\mathcal{A}$, and $Br(W)$ the generic Brauer-Chen algebra defined over $\kk$,
as in section \ref{sect:presentations}. Let us pick a representative $H_0$ for each $\mathcal{A}_0 \in \mathcal{A}/W$. Then $Br_1(W)$ is a semisimple algebra, whose irreducible representations not factoring through $\Qb W$ are in 1-1 correspondence with the pairs
$(\mathcal{A}_0,\theta)$ where $\mathcal{A}_0\in \mathcal{A}/W$ and $\theta$
is an irreducible representation of $N(W_{H_0})/W_{H_0}$ over $\kk$.
The restriction of such a representation to $W \subset Br_1(W)$ is the induced representation $\Ind_{N(W_{H_0})}^W \theta$.

The dimension of $Br_1(W)$ is equal to 
$$
|W| + \sum_{\mathcal{A}_0 \in \mathcal{A}/W} |\mathcal{A}_0| \times |W|/|W_{H_0}|
$$ 

\end{theorem}

The above result provides a complete description of $Br(W)$ exacly when $Br(W) = Br_1(W)$, that is when there is no pair of transverse reflecting hyperplanes. This is the case for all irreducible rank 2 groups, but also for a few groups in rank 3. Specifically, this is the case for the groups $G(e,e,3)$ with $e \geq 3$ (see \cite{NEAIMETH}, Proposition 5.3.2) as well as for the primitive reflection groups $G_{24}$ and $G_{27}$
(but it is \emph{not} the case for the Coxeter group $H_3$). Therefore, we get the following. 
\begin{corollary} (See section \ref{subsect:Gee3}) When $W = G(e,e,3)$ and $e$ is odd, $Br(W)$ is semisimple and has dimension $3e^2(2+3e)$.
\end{corollary}

In particular, this provides additional evidence towards conjecture 5.4.2 in \cite{NEAIMETH} about the structure of a generalized BMW algebra in type $G(e,e,3)$.

\medskip
 
A major obstacle to extend the above result to the whole algebra $Br(W)$ is that a good understanding in the general case is still missing of what was called in \cite{CFW}, for type ADE, \emph{admissible collections} of transverse hyperplanes, and described there in terms of the associated root system.

\subsection{Extensions and deformations}

Our second main result (see theorem \ref{theo:flat}) is the following one. We proved in \cite{YH,YH2} that Hecke algebras admit natural extensions by the M\"obius algebra $\C \mathcal{L}$ of the lattice $\mathcal{L}$ of the reflection subgroups of $W$, and that these algebras are monodromic deformation of $\C W \ltimes \C \mathcal{L}$ in the same way as the Hecke algebra is a monodromic deformation of $\C W$. Here we prove
that the same phenomenon occurs for the Brauer-Chen algebra. In particular, there is a natural KZ-type connection on $Br(W) \ltimes \C \mathcal{L}$ that `covers' in some sense these two different constructions, and which should be related, when $W = \mathfrak{S}_n$, with the tied-BWM algebra introduced by Aicardi and Juyumaya in \cite{AICJUYU}. In the framework of links invariants, this tied-BMW algebra
supports the Markov trace responsible for the Kauffman and HOMFLY
as well as their `tied' variants.

Our work then suggests that there should exist a general definition of a generalized (tied-)BMW algebra for arbitrary complex reflection groups, that should be a module of finite type over its natural ring of definition. An idea of how much
this module structure could be complicated for specific parameters is given e.g. by the extensive work of Yu on a BMW algebra for the family $G(d,1,n)$, see \cite{YU}.

\subsection{Brauer algebras and Krammer representations}

The very first prototype of the representations of $Br_1(W)$ described here were implicitely described in \cite{KRAMCRG}, as they served as a guide to Chen to construct his algebra. These representations can be deformed using monodromy means to the generalized Krammer representations of the braid group of $W$ in type ADE as defined in \cite{CW,DIGNE}, see \cite{KRAMMINF}. The explicit (algebraic, combinatorial) construction of these generalized Krammer representation for arbitrary complex reflection groups is quite an ongoing challenge. To my knowledge, the first successful attempt in this direction
in the non-real case has been made by Neaime in \cite{NEAIMETH}, where such a matrix representation was constructed
for $W = G(3,3,3)$, using a deformation of the Brauer-Chen algebra. It appears in this case that this representation admits non-trivial Galois conjugates. These conjugate representations are deformations of other level 1 representations of the Brauer-Chen representation, and this provides another motivation for the present work, namely to describe infinitesimally the Galois conjugates of the generalized Krammer representation that we introduced in \cite{KRAMCRG}. Note that Neaime also constructed a conjectural matrix model for $G(4,4,3)$.

In the same way that understanding the general structure of the BMW algebra is useful to deal with the Krammer representation, the results of the present paper should thus provide similar services in the ongoing process of understanding and constructing these generalized Krammer representations.

{\bf Acknowledgements.} I thank F. Digne and N. Matringe for useful discussions, G. Neaime for useful comments, and the anonymous referee for detailed remarks.

\section{The Brauer-Chen algebra and its extensions}

\subsection{Presentations}
\label{sect:presentations}

Let $W$ be a complex (pseudo-)reflection group, $\mathcal{R}$ its
set of reflections, and $\mathcal{A}$ the set of reflecting hyperplanes. 
Two hyperplanes $H_1,H_2$
are called transverse if $H_1 \neq H_2$ and
$\{ H \in \mathcal{A} \ | \ H \supset H_1 \cap H_2 \} = \{ H_1, H_2 \}$. In this case we write $H_1 \captr H_2$. A \emph{transverse collection}
of hyperplanes is a subset $\underline{H} = \{ H_1,\dots H_r \}$ of $\mathcal{A}$ with the property that every two hyperplanes in the collection are transverse. 
Let $\kk$ be a commutative ring with $1$, and $\delta \in \kk^{\times}$ an invertible scalar. We fix the choice of one scalar parameter $\mu_s \in \kk$ per reflection $s$, subject to the condition $\mu_s = \mu_t$ as soon as $s$ and $t$ are conjugate. We denote by $\underline{\mu}$ the collection of such parameters.

By definition, 
the Brauer-Chen algebra $Br(W) = Br_{\underline{\mu}}(W)$ attached to $\underline{\mu}$ is
defined by generators $w \in W$, $e_H, H \in \mathcal{A}$ together with the relations of $W$,
the semi-direct type relations $w e_H = e_{w(H)} w$ and the following relations
\begin{enumerate}
\item $\forall H \in \mathcal{A} \ \  e_H^2 = \delta e_H \ \ \& \ \ e_Hs = s e_H = e_H$ whenever $\Ker(s-1)=H$
\item $H_1 \captr H_2 \Rightarrow e_{H_1} e_{H_2}= e_{H_2} e_{H_1}$
\item If $H_1$ and $H_2$ are distinct and not transverse, then
$$
e_{H_1} e_{H_2} = \left(\sum_{s \in \mathcal{R} \ | \ s(H_2)=H_1 } \mu_s s \right)  e_{H_2}=e_{H_1} \left(\sum_{s \in \mathcal{R} \ | \ s(H_2)=H_1 } \mu_s s \right)
$$
\end{enumerate}
Notice that the two equalities in relation (3) can be deduced one from the
other using the semi-direct type relations.
\begin{remark}
A natural attempt to a further generalization would be to consider relations of the form $e_H^2 = \delta_H e_H$ for scalars $\delta_H$ depending on $H$. By the semi-direct type relations such a scalar should depend only on the orbit of $H$ under $W$. But then, up to rescaling the coefficients $\underline{\mu}$ we can assume all such $\delta_H$'s are the same, at least when all of them  are nonzero. Moreover, if some of them are zero this provides quotients of the orginial algebra by relations of the form $e_H^2$ so this also fits inside the original framework.
\end{remark}

\subsection{Chen's additional relations}

In the relations above, we removed one relation from Chen's original definition, the relation called (1)' in \cite{CHENBR}. This relation is that
$w e_H = e_H w = e_H$ if $w(H)=H$ and there exists  $H_1, H_2 \in \mathcal{A}$
such that $\Ker(w-1)\cap H = H_1\cap H_2$ and $H_1,H_2$ are not transverse.

It is claimed without proof
in \cite{CHENBR} that this condition is equivalent to the following one
$$
(1)'' \ \ \ s e_H = e_H s = e_H \ \ \mbox{ if }
s \in \mathcal{R}, s(H) = H  \mbox{ and } H \mbox{ and } \Ker(s-1)  \mbox{ are (distinct and) not transverse. } 
$$

Actually, the preprint version of \cite{CHENBR} on the arxiv (\verb+arXiv:1102.4389v1+) is the version of (1)' given here, and the claim that (1)' and (1)'' are equivalent can be found only in the published version. But on the other hand, in the published version of (1)' the
element $w$ is additionally assumed to be a reflection, which makes (1)' not only \emph{equivalent}, but \emph{formally equal} to (1)''. So we have to interprete this addition to be a typo (corroborated by some examples in Chen's paper).

The implication $(1)' \Rightarrow (1)''$ is because, if $s \in \mathcal{R}$ is such that $s(H)=H$ and $H$ and $H_s = \Ker(s-1)$ are not transverse, then, setting $H_1 = H_s$, $H_2 = H$ and $w = s$, we get $\Ker(w-1)\cap H = H_1 \cap H_2$ and by $(1)'$ the conclusion.
Conversely, if $w,H_1,H_2$ are as in the assumption of $(1)'$, then $w$ belongs to the parabolic subgroup fixing $H_1 \cap H_2$. If this rank 2 parabolic subgroup is a dihedral group,
then we have the conclusion because
\begin{enumerate}
\item either $w$ is a reflection, and we have the conclusion by (1)''
\item otherwise, setting $s$ the reflection w.r.t. $H$,
we have that $ws$ is a reflection satisfying the same
assumptions, and writing $w = ws.s$ we get the conclusion
by applying (1)' twice.
\end{enumerate}

Therefore, (1)' and (1)'' are equivalent in a number of cases, including all Coxeter groups. It is however not
true, in general that (1)'' implies (1)'. In order to check this, we consider the complex reflection group $W$ of type $G(4,2,2)$ (see section \ref{sect:G422} below), and apply to the given presentations a Gr\"obner basis algorithms, using the GAP4 package GBNP (see \cite{GBNP}), for a given value
of $\delta \in \kk = \Q$. We obtain the dimensions $28$ and $40$
depending whether we add (1)' or (1)'', and $64$ without
both of them.

In \cite{CHENBR} \S 9, it is argued that a reason for relation (1)' to be added is that it may be closer to a
previously introduced algebra in Coxeter type $B$ and more
generally in type $G(d,1,n)$ (see \cite{HO}). Our purpose here being to consider the largest possible finite-dimensional algebra
we consider it better to eliminate this. As noted by Chen, this condition is void in the usual (type $A$) case,
so we get indeed a generalization of the usual Brauer algebra.

\subsection{Flat connection}
In any case, we reprove Proposition 5.1 of \cite{CHENBR} in order to make it clear that these additional relations are not needed for the associated connection to be flat. Actually, we notice that our relation (1) is not needed either. Therefore, we let $Br^0(W)$ denote the algebra defined as $Br(W)$ but with relation (1) removed.
We remark that this algebra is actually defined over $\Z[\underline{\mu}]$
and has infinite rank. To every hyperplane $H$ we associate the logarithmic 1-form $\om_H = (1/\pi \ii ) \dd \alpha_H/\alpha_H$ where $\alpha_H$ is an arbitrary linear form with kernel $H$. This 1-form is defined on the complement
of the hyperplane arrangement $\mathcal{A}$, that we denote $X$.

\begin{proposition} \label{prop:chen1form} (Chen) Assume $\kk = \C$. Then, the following 1-form
$$
\om = \sum_{H \in \mathcal{A}} \left(\left(\sum_{Ker(s-1)=H} \mu_s s \right)  - e_H\right) \om_H \in \Omega^1(X) \otimes Br^0(W)
$$
is integrable and $W$-equivariant.
\end{proposition}
\begin{proof} 
We let $t_H = \left(\sum_{Ker(s-1)=H} \mu_s s \right)  - e_H$. It is clear that $w t_H w^{-1} =  t_{w(H)}$ hence we only need to prove
that Kohno's holonomy relations of \cite{KOHNO} are satisfied. We recall these relations now. Let $Z$ be a codimension 2 flat, and $t_Z = \sum_{H \supset Z} t_H$. One needs to prove $[t_Z,t_H]=0$ for all $H \supset Z$. If $Z$ is what Chen calls a crossing edge, that is if it is contained in exactly two hyperplanes $H_1$ and $H_2$, then it is clear that all the elements involved in $t_{H_1}$ and $t_{H_2}$ commute with each other, whence $[t_{H_1},t_{H_2}]=0$ hence 
$[t_Z,t_H]=0$ for all $H \supset Z$.

If not, letting $\varphi_H = \sum_{Ker(s-1)=H} \mu_s s$ and $\varphi_Z = \sum_{H \supset Z} \varphi_H$, we first notice
that $[\varphi_H,\varphi_Z]=0$ from the integrability of the Cherednik connection (see e.g. \cite{BMR}). Let us pick $H_0 \supset Z$.
We have
$$
[t_{H_0},t_Z]=
[\varphi_{H_0} - e_{H_0},\varphi_Z - \sum_H e_H]=
-[\varphi_{H_0},\sum_H e_H] - [e_{H_0},\varphi_Z - \sum_H e_H].
$$ 
But, for all $s \in \mathcal{R}$ with $\Ker(s-1) \supset Z$, we have $s (\sum_H e_H) = (\sum_H e_H) s$ hence $[\varphi_{H_0},\sum_H e_H]=0$.
It remains to compute
$
[e_{H_0},\varphi_Z - \sum_H e_H]$.
We have
$$
e_{H_0}\sum_H e_H = \sum_H e_{H_0} e_H
= \sum_H e_{H_0} \sum_{s(H) = H_0} \mu_s s 
= e_{H_0} \sum_{s \in \mathcal{R} ; \Ker(s-1) \supset Z}
  \mu_s s 
  = e_{H_0} \varphi_Z
$$
and similarly  $(\sum_H e_H) e_{H_0}=\varphi_Z e_{H_0}$
hence 
$$
[e_{H_0},\varphi_Z - \sum_H e_H]=0
$$
and this proves the claim.

\end{proof}

\subsection{Finiteness of dimension}

Chen proved that his algebra has finite rank as a $\kk$-module. Actually, following the same lines of proof, one can prove the following more general
statement.

\begin{proposition} \label{prop:findimchen}
Let $\kk$ be a commutative ring with $1$,
 $\underline{Q}=(Q_H)_{H \in \mathcal{A}}$
 a family of polynomials in one indeterminate over $\kk$
 such that $H_1 \sim H_2 \Rightarrow Q_{H_1}=Q_{H_2}$. Then the quotient 
 $Br^{\underline{Q}}(W)$ of $Br^0(W)$ by the relations $e_HQ_H(e_H)=0$ has finite rank over $\kk$.
\end{proposition}
\begin{proof}
By the semidirect product relations it is clear that every element $Br^0(W)$ is a linear combinations of terms of the form $w e_{H_1} \dots e_{H_r}$ for $w \in W$. 
If $m$ is the maximum of the degrees of the $Q_H$, we claim that we need no term with $r$ larger than $m \rk(W)$.
 We argue by contradiction, and consider a term with $r$ minimal but greater than $m \rk(W) $ which can not be rewritten using smaller $r$'s. First of all, for every $i < r$ we have that the $H_i$ and $H_{i+1}$ are either equal or transverse, by minimality of $r$ and relation (3). In particular, $e_{H_i}$ commutes with $e_{H_{i+1}}$, and actually (by induction) with every $e_{H_j}$ for $j > i$ for the same reason. Therefore we can
assume that 
$$
(H_1,\dots, H_r) = (\underbrace{J_1,\dots,J_1}_{u_1},\dots , \underbrace{J_k,\dots,J_k}_{u_k})
$$
with $\{J_1,\dots, J_k \}$ of cardinality $k$, and in particular a transverse collection of hyperplanes. By the polynomial relation on the $e_H$'s and the minimality of $r$, we have $u_i \leq  m$ for all $i$, hence $r \leq m \times k$.
Let us pick for each $i$ a nonzero vector $v_i$ in the
orthogonal complement of $J_i$ (with respect to some unitary form preserved by $W$). For $i \neq j$ we have that $J_i$ and $J_j$ are transverse, and thus $v_i$ and $v_j$ are orthogonal. Therefore the $v_i$ form an orthogonal family of cardinality $k$, thus $k \leq \rk(W)$ and $r \leq m  \rk(W)$, a contradiction. This proves the claim. 
\end{proof}

When $\kk$ is a field, it is clear that all finite-dimensional representations of $Br^0(W)$
factorize through $Br^{\underline{Q}}(W)$
for some $\underline{Q}$. Moreover, note that,
picking one root of $Q_H$ for each conjugacy class of hyperplanes provide a surjective morphism from $Br^{\underline{Q}}(W)$ to $Br(W)$
with parameter(s) corresponding to the root(s).
Actually, from the proof of the classification of the irreducible
representations of $Br_1(W)$ given below,
it will be clear that all the irreducible
representations of $Br_1^{\underline{Q}}(W)$
(over an algebraically closed field) factor through one of them.

\subsection{A generalized Vogel algebra}

Finally, we prove here that, when $W$ is a 2-reflection group, the Brauer-Chen
algebra appears as a quotient of an algebra $\mathfrak{Q}(W)$ depending on two scalar parameters $\alpha,\beta$, and defined by generators and relations as follows. Generators are $t_H, H \in \mathcal{A}$, $w \in W$, and the 
relations are the relations of $W$ together with the $W$-invariance and holonomy relations
$$
w t_H = t_{w(H)} w, \left[ t_H,\sum_{H'\supset Z} t_{H'}\right]=0
$$
where $Z$ runs among the codimension 2 flats, and
$$
\left\lbrace \begin{array}{l}
t_H s_H = s_H t_H = t_H \\
t_H^2 - (\alpha+\beta) t_H + \frac{\alpha \beta}{2}(1 + s_H) = 0 \end{array} \right.
$$
for all $H \in \mathcal{A}$. This algebra is a generalization of an algebra introduced by P. Vogel in
the framework of Vassiliev invariants, see \cite{GQ}.

Since $W$ is a 2-reflection group, in the definition of the Brauer-Chen algebra one can set $\la_H = \mu_{s_H}$ for $s_H$ the only reflection with hyperplane $H$.
Let us set $t_H = \la_H(1+s_H) -e_H \in Br(W)$. Then it is readily checked that $t_H s_H = t_H$,
and $t_H^2 = 2\la_H^2(1+s_H) + \delta e_H - 4 \la_H e_H
= 2\la_H^2(1+s_H) + (\delta - 4 \la_H) e_H$. Then $t_H^2 + (\delta-4 \la_H) t_H
= 2\la_H^2(1+s_H) + (\delta - 4 \la_H) e_H
+(\delta-4 \la_H)\la_H(1+s_H) -(\delta-4 \la_H)e_H
=   
(\delta-2 \la_H)\la_H(1+s_H) $, hence 
$t_H^2 - (\alpha+\beta) t_H + \frac{\alpha \beta}{2}(1 + s_H) = 0$ with $\alpha = 2 \la_H -\delta$
and $\beta = 2 \la_H$.

This proves the following.
\begin{proposition}
If $W$ is a 2-reflection group and $\mu_{s_H} = \la_H$,
$\alpha = 2 \la_H -\delta$
and $\beta = 2 \la_H$,
then there exists a surjective morphism
$\mathfrak{Q}(W) \onto Br(W)$ which is the identity on $W$
and maps $t_H \mapsto \la_H(1+s_H)-e_H$.
\end{proposition}

An intriguing open question is whether these algebras $\mathfrak{Q}(W)$ are finite dimensional in general. It is conjectured to be the case when $W = \mathfrak{S}_n$, and known to be true for $n \leq 5$ by \cite{GQ}.

When $W$ is a finite Coxeter group with generating set $S$, another presentation of $\mathfrak{Q}(W)$
is easily seen to be given by generators $t_H = t_{s_H}, H \in \mathcal{A}$, $s \in S$, together with the Coxeter relations, the holonomy relations, $s t_u = t_{sus} s$ for $s\in S$
and $u\in \mathcal{R}$, 
and
$$
\left\lbrace \begin{array}{l}
t_s s = s t_s = t_s \\
t_s^2 - (\alpha+\beta) t_s + \frac{\alpha \beta}{2}(1 + s) = 0 \end{array} \right.
$$
for all $s \in S$.

\subsection{Lattice extensions}

We operate a mixture of these ideas together with
the ones of \cite{YH,YH2}. Let $\mathcal{L}$ denote an admissible lattice in the sense of \cite{YH2}, that is a $W$-invariant sublattice of the lattice of all 
full reflection subgroups of $W$, where full means that if $s \in W$ is a reflection, then $W$ contains all the
reflections fixing $\Ker(s-1)$, with
the following properties of containing the
cyclic (full) reflection subgroups and the trivial subgroups.

We consider its Möbius algebra $\kk \mathcal{L}$, and denote
$f_L, L \in  \mathcal{L}$ the natural basis element. Recall from \cite{YH2} that full reflection subgroups can be naturally indexed by the collection of their reflection hyperplanes, and so in particular we denote $f_H \in \kk \mathcal{L}$, for $H \in \mathcal{A}$, the natural basis element
associated to the full reflection subgroup fixing $H$. 

The following proposition was proved in \cite{YH} under
the additional unnecessary assumption that $\mathcal{L}$ is the lattice of all parabolic subgroups. Here we provide the
general proof.

\begin{proposition} \label{prop:integrYH} Assume $\kk = \C$. Let us choose a collection of scalars $\la_H, H \in \mathcal{A}$ such that $\la_{w(H)} = \la_H$ for all $w \in W$. Then the following 1-form
$$
\om = \sum_{H \in \mathcal{A}} \left(\la_H + \sum_{Ker(s-1)=H} \mu_s s \right)   f_H \om_H \in \Omega^1(X) \otimes \C W\ltimes \C\mathcal{L}
$$
is integrable and $W$-equivariant over the hyperplane complement $X$.
\end{proposition}
\begin{proof} 
Let us assume that we have picked a `distinguished' reflection $s_H$ for each $H \in \mathcal{A}$ with the property that $w s_H w^{-1} = s_{w(H)}$ for all $w \in \mathcal{A}$ and $\langle s_H \rangle = W_H$. We denote $\la_H^{(0)} =  \la_H$ and $\la_H^{(i)} = \mu_{s_H}^i$ for $i > 0$ and $H \in \mathcal{A}$. We set $t_H =  \sum_{0\leq i < m_H} \la_H^{(i)} s_H^i f_H$, with $m_H$ the order of $s_H$.
Let $Z$ be a codimension 2 flat, and $t_Z = \sum_{H \supset Z} t_H$. One needs to prove $[t_Z,t_{H_0}]=0$ for all $H \supset Z$. For this we only need to prove that
$[s f_{H_0},t_Z]=0$ for $s \in \langle s_H \rangle$.
We do this. We have
$$
s f_{H_0} t_Z = \sum_{\stackrel{H \supset Z}{0 \leq i < m_H}} 
s f_{H_0} \la_H^{(i)} s_H^i f_H
= \sum_{\stackrel{H \supset Z}{0 \leq i < m_H}}\la_H^{(i)} 
ss_H^i f_{s_H^{-i}(H_0)}   f_H  
$$
and
$$
t_Z s f_{H_0}=\sum_{\stackrel{H \supset Z}{0 \leq i < m_H}} \la_H^{(i)} s_H^i f_Hs f_{H_0} =
\sum_{\stackrel{H \supset Z}{0 \leq i < m_H}} \la_H^{(i)} s (s^{-1}s_H s)^if_{s^{-1}(H)} f_{H_0} 
$$
{}
$$
=
\sum_{\stackrel{H \supset Z}{0 \leq i < m_H}} \la_{s^{-1}(H)}^{(i)} s s_{s^{-1}(H)}^i f_{s^{-1}(H)} f_{H_0} 
=
\sum_{\stackrel{H \supset Z}{0 \leq i < m_H}} \la_{H}^{(i)} s s_{H}^i f_{H} f_{H_0} 
$$
and so we only need to check that $f_{H} f_{H_0} = 
f_{s_H^{-i}(H_0)}   f_H$. It is sufficient to prove this for the maximal admissible lattice $\mathcal{L}=\mathcal{L}_{\infty}$ of all full reflection subgroups. In this case, $f_H f_{H_0} = f_{G}$
for $G$ the smallest full reflection subgroup containing $\langle s_H, s_{H_0}\rangle$.
Since $\langle s_{s_H^{-i}(H_0)},s_H \rangle=
\langle s_H^{-i}s_{H_0}s_H^{i},s_H \rangle
=\langle s_{H_0},s_H \rangle$ we indeed get 
$f_{H} f_{H_0} = 
f_{s_H^{-i}(H_0)}   f_H$, and this proves the claim, $W$-invariance being obvious.

\end{proof}

 We now denote $Br^0(W,\mathcal{L})$ the algebra presented by generators $w \in W$, $e_H, H \in \mathcal{A}$, $f_L, L \in \mathcal{L}$, together
with the relations
\begin{itemize}
\item $w = w_1w_2$ if $w = w_1w_2$ inside $W$
\item $w e_H = e_{w(H)}w$
\item $H_1 \captr H_2 \Rightarrow e_{H_1} e_{H_2}= e_{H_2} e_{H_1}$
\item If $H_1$ and $H_2$ are distinct and not transverse, then
$$
e_{H_1} e_{H_2} = \left(\sum_{s \in \mathcal{R} \ | \ s(H_2)=H_1 } \mu_s s f_{H_s} \right)  e_{H_2}=e_{H_1} \left(\sum_{s \in \mathcal{R} \ | \ s(H_2)=H_1 }\mu_s sf_{H_s} \right)
$$
\item $e_H f_L = f_L e_H$
\item $e_H f_H = f_H e_H = e_H$ 
\item $f_{L_1} f_{L_2} = f_{L_1 \vee L_2}$
\end{itemize}

The following then provides an upgrading of both Propositions \ref{prop:chen1form} and \ref{prop:integrYH}.

\begin{theorem}
\label{theo:flat} 
Assume $\kk = \C$. Let us choose a collection of scalars $\la_H \in \C$, $H \in \mathcal{A}$ such that $\la_{w(H)} = \la_H$ for all $w \in W$. Then the following 1-form
$$
\om = \sum_{H \in \mathcal{A}} \left(\left(\la_H + \sum_{Ker(s-1)=H} \mu_s s \right)  - e_H\right) f_H \om_H \in \Omega^1(X) \otimes Br^0(W,\mathcal{L})
$$
is integrable and $W$-equivariant.
\end{theorem}
\begin{proof}
As in the proof of Proposition \ref{prop:chen1form},
we start by setting $\varphi_H = \la_H + \sum_{Ker(s-1)=H} \mu_s s$, $t_H = (\varphi_H - e_H)f_H$. It is clear that $w t_H w^{-1} =  t_{w(H)}$ hence we only need to prove
that Kohno's relations are satisfied. Let $Z$ be a codimension 2 flat, and $t_Z = \sum_{H \supset Z} t_H$. One needs to prove $[t_Z,t_H]=0$ for all $H \supset Z$. If $Z$ is a crossing edge, that is it is contained in exactly two hyperplanes $H_1$ and $H_2$, then it is clear that all the elements involved in $t_{H_1}$ and $t_{H_2}$ commute with each other, whence $[t_{H_1},t_{H_2}]=0$ hence 
$[t_Z,t_H]=0$ for all $H \supset Z$.

If not, letting $\psi_Z = \sum_{H \supset Z} \varphi_Hf_H$, let us pick $H_0 \supset Z$. We first notice that $[\varphi_Hf_H,\psi_Z]=0$ from
Proposition \ref{prop:integrYH}.

We have
$$
[t_{H_0},t_Z]=
[\varphi_{H_0}f_{H_0} - e_{H_0},\psi_Z - \sum_H e_H]=
-[\varphi_{H_0}f_{H_0},\sum_H e_H] - [e_{H_0},\psi_Z - \sum_H e_H]
$$ 
But, for all $s \in \mathcal{R}$ with $H_0=\Ker(s-1) \supset Z$, we have 
$$
sf_{H_0} (\sum_H e_H) =
f_{H_0}s (\sum_H e_H) = f_{H_0}(\sum_H e_H) s
= (\sum_H e_H) sf_{H_0}
$$ hence $[\varphi_{H_0},\sum_H e_H]=0$.
It remains to compute
$
[e_{H_0},\psi_Z - \sum_H e_H]$.
We have
$$
e_{H_0}\sum_H e_H = \sum_H e_{H_0} e_H
= \sum_H e_{H_0} \sum_{s(H) = H_0} \mu_s f_{H_s} s 
= e_{H_0} \sum_{s \in \mathcal{R} ; \Ker(s-1) \supset Z}
  \mu_s f_{H_s}s 
  = e_{H_0} \left(\psi_Z - \sum_{H\supset Z} \la_H f_H\right)
$$
and similarly  $(\sum_H e_H) e_{H_0}=(\psi_Z-\sum_{H\supset Z} \la_H f_H) e_{H_0}$
hence 
$$
[e_{H_0},\psi_Z - \sum_H e_H]=0
$$
and this proves the claim.
 
\end{proof}

Note that the quotient of the $Br^0(W,\mathcal{L})$ by the relations $e_H = 0$ provides the semidirect product $\kk W \ltimes \kk \mathcal{L}$, while the quotient by the relations $f_H = 1$
provides $Br^0(W)$. We can similarly
introduce the algebras $Br^{\underline{Q}}(W,\mathcal{L})$ and in particular $Br(W,\mathcal{L})$ by imposing
the relations $e_H^2 = \delta e_H$ for $H \in \mathcal{A}$.
By a straightforward adaptation of its proof, one gets the following analog of Proposition \ref{prop:findimchen}.

\begin{proposition} \label{prop:findimextchen}
Let $\kk$ be a commutative ring with $1$,
 $\underline{Q}=(Q_H)_{H \in \mathcal{A}}$ a family of polynomials in one indeterminate over $\kk$
 such that $H_1 \sim H_2 \Rightarrow Q_{H_1}=Q_{H_2}$. Then the quotient 
 $Br^{\underline{Q}}(W,\mathcal{L})$ of $Br^0(W,\mathcal{L})$ by the relations $e_HQ_H(e_H)=0$ has finite rank over $\kk$.
\end{proposition}

The existence of the flat connection of Theorem \ref{theo:flat} raises the
following question :

\begin{question}
When $W = \mathfrak{S}_n$ and $\mathcal{L}$ is the lattice of all reflection subgroups of $W$, does this monodromy representation of the braid group over $Br(W,\mathcal{L})$ provide the braid group representations factoring through the tied-BMW algebra of Aicardi and Juyumaya (see \cite{AICJUYU}) ?
\end{question}

\section{The ideals $I_r$}

 To any transverse collection
$\underline{H} = \{ H_1,\dots, H_r \}$ we associate $e_{\underline{H}} = e_{H_1}\dots e_{H_r} \in Br(W)$. We denote $I_r$ for $r \geq 1$ the left ideal of $Br(W)$ generated by the $e_{\underline{H}}$ for $\underline{H}$ of cardinality $r$ (or equivalently, at least $r$).
We first prove

\begin{lemma}
Each $I_r$ is a two-sided ideal of $Br(W)$.
\end{lemma}
\begin{proof}
Let $\underline{H}$ be a transverse collection of
cardinality $r$. One needs to prove that $e_{\underline{H}} w \in I_r$ and
$e_{\underline{H}} e_{K} \in I_r$ for all $w \in W$ and $K \in \mathcal{A}$. One readily checks that $e_{\underline{H}} w = e_{w^{-1}(\underline{H})}$ and that $w^{-1}(\underline{H})$
is again a transverse collection, so this proves 
$e_{\underline{H}} w \in I_r$
and more generally $e_{\underline{H}}\Q[\delta] w \subset I_r$. We now consider
$e_{\underline{H}} e_{K}$. If $K \in \underline{H}$ then clearly $e_{\underline{H}} e_{K} = \delta e_{\underline{H}} \in I_r$.  Otherwise, if $K$ is transverse to every hyperplane in $\underline{H}$ we have
$e_{\underline{H}} e_{K}= e_{\underline{H} \cup \{ K \}} = e_K e_{\underline{H}} \in I_r$. If not,
then $\underline{H} = \underline{J} \cup \{ L \}$
with $L$ not transverse to $K$. But then
$e_L e_K \in e_L.\Q[\delta] W$ hence
$e_{\underline{H}} e_K \subset e_{\underline{H}}\Q[\delta] W \subset I_r$ as we already proved. This proves the claim.
\end{proof}

We denote $Br_r(W) = Br(W)/I_{r+1}$, so
in particular $Br_0(W) = \Q[\delta ] W$.

For any transverse collection $\underline{H}$,
let $W_{\underline{H}}$ denote the subgroup generated by all the reflections w.r.t. a hyperplane in $\underline{H}$.
For $w \in W_{\underline{H}}$, we have
$w e_{\underline{H}} = e_{\underline{H}}$,
hence for arbitrary $w \in W$ the element
$w e_{\underline{H}}$ depends only on the class
of $w$ modulo $W_{\underline{H}}$. Therefore,
$c e_{\underline{H}}$ is well-defined for $c \in W/W_{\underline{H}}$.

\begin{proposition} \label{prop:spannBr}
The collection of the $c W_{\underline{H}}$ for
$\underline{H}$ a (possibly empty) transverse collection of hyperplanes and $c \in W/W_{\underline{H}}$ is a spanning set for $Br(W)$.
\end{proposition}
\begin{proof}
By the above remark it is equivalent to say that the collection of the $w e_{\underline{H}}$ for
$w \in W$ and $\underline{H}$ a transverse collection form a spanning set.
For this we need to prove that multiplying on the
left such elements by $g \in W$ and $e_K$ for $K \in \mathcal{A}$ can be written as a linear combination of such elements. Since this is clear
for $g \in W$, we look at $e_K w e_{\underline{H}} = w e_{w^{-1}(K)} e_{\underline{H}}$ and we finally need to prove that $e_J e_{\underline{H}}$ is equal to such a linear combination. The proof is then similar to the one of the previous lemma : if $K$ is transverse to all hyperplanes inside $\underline{H}$
then $e_K e_{\underline{H}} =  e_{\underline{H} \cup \{ K \}}$ and we are done, the case $K \in \mathcal{H}$ is also clear, and 
otherwise we can write $\underline{H} = \{ L \} \cup \underline{J}$ with $L$ not transverse to $K$ and 
$$e_K e_{\underline{H}} = e_K e_L e_{\underline{J}} \in \Q[\delta ]W e_L e_{\underline{J}}\subset \Q[\delta ] W e_{\underline{H}}$$
and we are done.
\end{proof}

As a corollary, when $W$ is a 2-reflection group admitting a single conjugacy class of reflections, then $Br_1(W)$ is spanned by a family of cardinality
$$
|W| + |\mathcal{R}|\times |W|/2
$$

Now consider the case of a Coxeter group with generating set $S$. One could consider the
ideal $J_2$ generated by the $e_r e_s$ for $r,s \in S$ being non-adjacent nodes of the Coxeter diagram. It is clear that $J_2 \subset I_2$. Conversely, if $\underline{H}= (H_1,H_2)$ is a transverse collection, the parabolic subgroup $W_Z$ fixing $Z = H_1 \cap H_2$ is generated by $s_{H_1}$ and $s_{H_2}$. It is conjugate to a standard parabolic $\langle r_1,r_2 \rangle$
with $r_1,r_2 \in S$, $s_{H_i} = w r_i w^{-1}$
for $i = \{1,2\}$ and some $w \in W$. Since $\Ker(r_1-1)$ and $\Ker(r_2 - 1)$ are also transverse we have that $r_1$ and $r_2$ are non-adjacent in the Coxeter graph hence $e_{r_1} e_{r_2} \in J_2$. But then $e_{H_1}e_{H_2 } = w e_{r_1}e_{r_2}w^{-1} \in J_2$ hence $I_2 \subset J_2$ and this proves that $I_2=J_2$. This provides a sometimes more handy description of $Br_1(W)$ in the Coxeter case.

\section{Connections with the Cohen-Frenk-Wales algebra}

In \cite{CFW}, Cohen, Frenk and Wales associated a Brauer algebra to any Coxeter group of type ADE as follows. If $(W,S)$ is such a Coxeter system, it is defined by generators $r \in S$, $e_r, r \in S$
and the Coxeter relations on $S$ together with
the following ones, where the notation $r \sim s$ means that $r$ and
$s$ are connected in the Coxeter diagram :

\begin{itemize}
\item (RSre) $r e_r = e_r$
\item (RSer) $ e_rr = e_r$
\item (HSee) $e_r^2 = \delta e_r$
\item (HCee) if $r \not\sim s$ then $e_r e_s = e_s e_r$
\item (HCer) if $r \not\sim s$ then $e_r s = s e_r$
\item (RNrre) if $r \sim s$ then $rse_r = e_se_r$
\item (HNrer) if $r \sim s$ then $se_rs = re_sr$ 
\end{itemize}

We denote this algebra by $BrCFW(W)$.
There is a natural injective map $W \to BrCFW(W)$ mapping $S$ to $S$ identically.

For every $t \in \mathcal{R}$ there is $s \in S$ and $w \in W$ such that $t = w s w^{-1}$.

We replace (RNrre) and (HNrer) by the following two conditions
\begin{itemize}
\item (RNrre') if $r \sim s$ then $rsre_s = e_r e_s$
\item (HNrer') if $r \sim s$ then $rsre_s = e_r srs$ 
\end{itemize}
Under the other conditions, it is readily checked that (RNrre) is equivalent to (RNrre') and that (HNrer) is equivalent to (HNrer'). From this it can be shown (see \cite{CHENBR}) that this algebra is isomorphic to the Brauer-Chen algebra $Br(W)$.

The irreducible representations on which $I_2$ vanishes
which are described in \cite{CFW} are indexed by a
$W$-orbit of positive roots -- which can be
identified to a $W$-orbit of reflecting hyperplanes -- together with an irreducible character of a subgroup (called $W(C)$ in \cite{CFW}) of $W$, which is shown (\cite{CFW}, Proposition 4.7) to be a complement of the parabolic subgroup $W_0$ fixing a given hyperplane $H_0$ inside the normalizer $N_W(W_0) = \{ w \in W \ | \ w(H_0)=H_0 \}$ of $W_0$, and is therefore isomorphic to $N(W_0)/W_0$. This complement is described as the reflection subgroup generated by the reflections associated to the roots orthogonal to the highest one in the $W$-orbit under consideration. The construction of the
representations is based on the root system. We show in the next section that these constructions can be made and generalized in a way independent of the chosen root system
to arbitrary complex reflection groups.

\section{Representations of $Br_1(W)$}

In this section we denote $\bkappa$ a field of characteristic $0$, and we assume $\kk = \bkappa(\delta)$
is the field of rational functions in $\delta$.

We will need the following easy lemma, for which we could not find a convenient reference.

\begin{lemma} \label{lem:descent} 
The map $M \mapsto M \otimes_{\bkappa} \kk$ induces a bijection between isomorphism classes of $\bkappa G$-modules and $\kk G$-modules.
\end{lemma}
\begin{proof} Recall (from e.g. prop. 1 of \cite{LIE456}, ch. V, annexe) the very general fact
that $M_1 \simeq M_2 \Leftrightarrow M_1 \otimes_{\bkappa} \kk
\simeq M_2 \otimes_{\bkappa} \kk$, so this map is injective. Since both ${\bkappa} G$ and $\kk G$ are semisimple algebras over their
base field, it is then sufficient to prove that every irreducible representation $\rho : G \to \GL_n(\kk)$ is isomorphic to the extension $\tilde{\rho}_0$ to $\kk$ of
a representation $\rho_0 : G \to \GL_n({\bkappa})$. Since $G$ is finite and ${\bkappa}$ is infinite there exists $\delta_0 \in {\bkappa}$ such that all entries of the $\rho(g)$, $g \in G$, viewed as rational functions, can be specialized at $\delta_0$ and such that $\delta_0$ is not a root of the $\det \rho(g)$, $g \in G$. Denote $\rho_0 : G \to \GL_n(\bkappa)$ the corresponding specialization. Now notice that the values $\tr \rho(g) \in \kk$ are algebraic over $\Q \subset {\bkappa}$. But all the elements of $\kk = {\bkappa}(\delta)$ which are algebraic over
${\bkappa}$ actually belong to ${\bkappa}$, hence $\tr \rho(g) \in \bkappa$ for all $g \in G$. It follows that the character of $\rho$ is equal to the character of $\tilde{\rho}_0$ hence $\rho \simeq  \tilde{\rho}_0$ and this proves the claim.
\end{proof}

\subsection{A direct presentation for $Br_1(W)$}

The algebra $Br_1(W)=Br(W)/I_2$ admits a
more tractable presentation with the same generators, and for relations the relations of $W$, the semidirect type relations, and the following ones :
\begin{itemize}
\item $(1)$\ $\forall H \in \mathcal{A} \ \  e_H^2 = \delta  e_H \ \ \& \ \ e_Hs = s e_H = e_H$ whenever $\Ker(s-1)=H$
\item $(3_1)$ If $H_1$ and $H_2$ are distinct, then
$$
e_{H_1} e_{H_2} = \left(\sum_{s \in \mathcal{R} \ | \ s(H_2)=H_1 } \mu_s s \right)  e_{H_2}=e_{H_1} \left(\sum_{s \in \mathcal{R} \ | \ s(H_2)=H_1 } \mu_s s \right)
$$
\end{itemize}
In other terms, the defining relations $(2)$ and $(3)$
of $Br(W)$ together with the defining relations of $I_2$ are replaced by $(3_1)$.
 
Indeed, relation $(2)$ modulo $I_2$ means $H_1 \captr H_2 \Rightarrow e_{H_1} e_{H_2}=0$, and this
is equivalent to asserting $(3_1)$ for $H_1$ and $H_2$, since
there are no reflection mapping $H_1$ to $H_2$ when $H_1$ and $H_2$ are transverse (see \cite{KRAMCRG}, Lemma 3.1).

\subsection{The $Br_1(W)$-modules of the form $\tilde{M}$ : definition.}

Let $H_0 \in \mathcal{A}$, $W_0 = W_{\underline{H}_0} = \langle s_0 \rangle$ be the pointwise stabilizer of $H_0$. We set $$
N_0 = \{ w \in W ; w(H_0) = H_0 \} = N_W(W_0)
$$
the normalizer of $W_0$ in $W$.
Let $M$ be a $\bkappa N_0$-module on which $W_0$ acts trivially. Then $\tilde{M} =  \kk W  \otimes_{\kk N_0} M_{\kk}$ with $M_{\kk}=\kk \otimes_{\bkappa} M$ is a $\kk W$-module which is the induced representation of $M_{\kk}$. 
Let $\mathcal{A}_0=W.H_0 \subset \mathcal{A}$ denote the orbit of $H_0$ under $W$. For all $H \in \mathcal{A}_0$ we pick $g_H \in W$ such that $g_H(H_0)=H$.
The set $\mathcal{G}_0$ of all such $g_H$ is
a set of representatives of $W/N_0$, hence $\tilde{M}$
admits a direct sum decomposition
$$
\tilde{M} = \bigoplus_{H \in \mathcal{A}_0} V_H $$
with $V_H = g_H \otimes M_{\kk}$ such that $w V_H \subset V_{w(H)}$ for all $w \in W$. Notice that $V_0= V_{H_0}$ is canonically identified to $M$ as an $N_0$-module, under
the natural inclusion $N_0 \subset W$.

We define $p_0 = p_{H_0} \in \End(\tilde{M})$ by
$$
p_{H_0} (x) = \sum_{u(H)=H_0} \mu_u u.x 
$$ 
where the $u$'s are understood to be reflections (and the $\mu_u$'s are the defining parameters of the algebra associated to them), and   $x \in V_H$, except when $H = H_0$ in which
case $p_{H_0} (x) = \delta x$.

\begin{lemma} \label{lem:p0tildeM} \ 
\begin{enumerate}
\item For all $w \in N_0$ and $x \in \tilde{M}$ we have $wp_0.x = p_0w.x$
\item $p_0(\tilde{M}) \subset V_{H_0}$
\item $p_0^2 = \delta p_0$ 
\end{enumerate}
\end{lemma}
\begin{proof}
We first prove (1). For $x \in V_{H_0}$ this is immediate. For $H \in \mathcal{A}\setminus \{ H_0 \}$ we have, for $w \in N_0$ and $y \in M$,
$$
wp_0.(g_H \otimes y) = w  \sum_{u(H)=H_0} \mu_u u.(g_H \otimes y)
=   \sum_{u(H)=H_0}\mu_u wug_H \otimes y
$$
while
$$
p_0w.g_H \otimes y =  p_0 .(wg_H)\otimes y= \sum_{vw(H)=H_0} \mu_v(vwg_H)\otimes y=
\sum_{wv^w(H)=H_0} \mu_v(wv^wg_H)\otimes y
$${}$$=
\sum_{v^w(H)=H_0} \mu_{v^w} wv^wg_H \otimes y
$$
and this proves the claim, as $v \mapsto v^w$ is a bijection of $\mathcal{R}$.
We now prove (2). For all $u$ with $u(H)=H_0$ we
have $u.(g_H \otimes y) = u g_H \otimes y$
and $ug_H(H_0) = u(H)=H_0$ hence $ug_H \in \C N_0$ and this proves the claim. (3) is then an immediate consequence.

\end{proof}

\begin{lemma}
For $w \in W$, and $H = w(H_0)$, then
$p_H \in \End(\tilde{M})$ defined by
$p_H.x= w p_0 w^{-1}.x$ depends only on $H$. If $H \in \mathcal{A}\setminus\mathcal{A}_0$ is not a conjugate of $H_0$ then we set $p_H = 0$. If $H_1,H_2 \in \mathcal{A}$ satisfy $H_2 = w(H_1)$ for some $w \in W$, then $w p_{H_1} w^{-1}.x = p_{H_2}.x$ for all $x \in \tilde{M}$.
\end{lemma}
\begin{proof}
Assume that $w_1,w_2 \in W$ satisfy $w_1(H_0)=w_2(H_0)$. Then $w = w_2^{-1} w_1 \in N_0$ hence $w p_0 w^{-1}.x = x$ for all $x \in \tilde{M}$ by the previous lemma (1)
and this implies $w_1 p_0 w_1^{-1}.x = w_2 p_0 w_2^{-1}.x  $ for all $x \in \tilde{M}$ and this proves the first part of the lemma. For the second part, if $H_2\not\in\mathcal{A}_0$ then $H_1 \not\in\mathcal{A}_0$ too,
hence $p_{H_1} = p_{H_2} = 0$ satisfy the property. If not, let $w_0  \in W$ such that $H_1 = w_0(H_0)$. Then
$H_2 = w(H_1)=(w w_0)(H_0)$. It follows that, for all $x \in \tilde{M}$, we have
$p_{H_2}.x = (w w_0) p_0(ww_0)^{-1}.x
= w (w_0 p_0w_0^{-1})w^{-1}.x
= w p_{H_1}w^{-1}.x$ and this proves the claim.
\end{proof}

A consequence of the definition is that, for $w(H_0)=H_1$, we have $$
p_{H_1}(\tilde{M}) = w p_{H_0} w^{-1}(\tilde{M})
= w p_{H_0} (\tilde{M})
= w V_{H_0} = V_{w(H_0)}
= V_{H_1}.
$$
Since, on $V_{H_2}$, $p_0$ coincides with the action
of $\sum_{u(H_2)=H_0} \mu_u u$, this implies that
$$
p_0p_{H_2} =  \left(\sum_{u(H_2)=H_0} \mu_u u\right) p_{H_2}
$$
for all $H_2 \in \mathcal{A}$. Then, 
since $H = w(H_0)$, we have
$$
p_{H_1} p_{H_2}=
p_{w(H_0)} p_{H_2}=
wp_{H_0}w^{-1} p_{H_2}=
wp_{H_0} p_{w^{-1}(H_2)}w^{-1}$$
which is equal by the above to
$$
 w \left(\sum_{vw^{-1}(H_2) = H_0} \mu_v vp_{w^{-1}(H_2)}\right)w^{-1}
 =  \sum_{vw^{-1}(H_2) = H_0} \mu_v wvw^{-1}p_{H_2}$${}$$
 =  \sum_{\, ^w v(H_2) = w(H_0)}\mu_{\, ^w v} \, ^w vp_{H_2}
 =  \sum_{\, u(H_2) = H_1} \mu_u up_{H_2}
$$
where as usual $u,v$ are assumed to belong to $\mathcal{R}$,
and this proves
$$
p_{H_1} p_{H_2}=\left(\sum_{u(H_2)=H_1} \mu_u u\right) p_{H_2}
$$
for all $H_1,H_2 \in \mathcal{A}_0$.

We now want to prove that 
$$
p_{H_1} p_{H_2}=p_{H_1} \left(\sum_{u(H_2)=H_1} \mu_u u\right).
$$
For this we note that
$$
p_{H_0} p_{H_1}=\left( \sum_{v(H_1)=H_0} \mu_v v \right) p_{H_1}
=\sum_{v(H_1)=H_0} vp_{H_1}v^{-1} \mu_v v  
=\sum_{v(H_1)=H_0} p_{H_0} \mu_v v  
=p_{H_0} \left( \sum_{v(H_1)=H_0}  \mu_v v\right)  
$$
and we conclude as before.

Finally, we need to check that $p_0 s = s p_0 = p_0$ whenever $\Ker(s-1)=H$. We have $p_0(\tilde{M}) = V_{H_0} = 1 \otimes M$, and $s.(1 \otimes y) = 
s \otimes y = 1 \otimes s.y$. But $W_0$ acts trivially on $M$ by assumption, hence $s p_0 = p_0$. Then $p_0 s = s (s^{-1} p_0 s) = s p_{s^{-1}(H_0)} = sp_0=p_0$, and through $W$-conjugation we get
$p_H s = s p_H = p_s$ for all $H \in \mathcal{A}_0$,
and $\Ker(s-1)=H$, the case $H \not\in \mathcal{A}_0$ being trivial.

This proves that $e_H \mapsto p_H$ extends the $\kk W$-module structure on $\tilde{M}$
to a $Br_1(W)$-module structure.

\subsection{The $Br_1(W)$-modules of the form $\tilde{M}$ : properties.}

\begin{proposition}
Let $M$ be a $\bkappa N_0$-module on which $W_0$ acts trivially. If $M$ is irreducible, then $\tilde{M}$ is an irreducible $Br_1(W)$-module. Moreover, if $M_1$ and $M_2$ are two such irreducible modules,
then $\tilde{M}_1 \simeq \tilde{M}_2$ iff $M_1\simeq M_2$.
\end{proposition}
\begin{proof}
For the irreducibility, we adapt the arguments of \cite{CFW}, Proposition 5.3.
Assume that $M$ is irreducible, and let $U \subset \tilde{M}$ a nonzero $Br_1(W)$-invariant subspace, and $q \in U \setminus \{ 0 \}$.

We have
$q = \sum_H \la_H g_H \otimes y_H$ for some $
y_H \in M_{\kk}$ and a collection of $\la_H \in \kk = \bkappa(\delta)$. Each $y_H \in M_{\kk}$ can be written
$y_H = \sum_i \nu_{H,i} m_{H,i}$ for some $m_{H,i} \in M$ and $\nu_{H,i} \in \kk$. Up to chasing denominators, we
can assume $\la_H, \nu_{H,i} \in \bkappa[\delta]$. 
Let us pick $H_1$ and $i_0$ with $\la_{H_1} \nu_{H,i_0}$ being of maximal degree in $\delta$ among the terms $\la_H \nu_{H,i} \in \bkappa[\delta]$ with $y_H \neq 0$. Then
$e_{H_1}.q = \sum_H \sum_i \la'_H \nu'_{H,i} g_{H_1} \otimes m'_{H,i}$ for some $m'_{H,i} \in M$ and $\la'_H, \nu'_{H,i}
\in \bkappa[\delta]$ with $\deg \la'_H \nu'_{H,i} \leq \deg
\la_H \nu_{H,i} \leq \la_{H_1} \nu_{H_1,i}$ whenever $H \neq H_1$. On the other hand, we have $\la'_{H_1} \nu'_{H_1,i}
 = \delta \la_{H_1} \nu_{H_1,i}$ and $m'_{H_1,i} = m_{H_1,i}$ for all $i$. It follows that $e_{H_1}.q$ can be written as
$$
g_{H_1}\otimes\left( \delta\la_{H_1}\left(\sum_i \nu_{H_1,i} m_{H_1,i}\right) + \sum_j \omega_j m''_j \right) = 
g_{H_1}\otimes \left( \delta g_H \otimes y_{H_1} + \sum_j \omega_j m''_j \right)
$$
with the $m''_j \in M$, $\omega_j \in \bkappa[\delta]$ with
degree (strictly) lower than maximal degree of the $\la_{H_1}\nu_{H_1,i}$. Since $y_{H_1} \neq 0$, this implies that $e_{H_1}.q \neq 0$.
On the other hand, $e_{H_1}.q \in U \cap V_{H_1}$ hence $U \cap V_{H_1} \neq \{ 0 \}$.
We can thus assume $q = g_{H_1} \otimes y \in U \cap V_{H_1} \setminus \{ 0 \}$.

Since $M$ is irreducible, by Lemma \ref{lem:descent}
we have that $M_{\kk}$ is irreducible as a $\kk N_0$-module. It follows that $V_{H_1} \subset U$.
Since $W$ acts transitively on the $V_H$, $H \in \mathcal{A}_0$, it
follows that $U$ contains all the $V_H$ hence $U = \tilde{M}$ and
$\tilde{M}$ is an irreducible $Br_1(W)$-module.

Now assume that a given $Br_1(W)$-module $\tilde{M}$ that
is obtained from two simple $\bkappa N_0$-modules with the property that the
action of $W_0$ is trivial. We thus have two a priori distinct direct sum decompositions
$$
\tilde{M} = \bigoplus_{H \in \mathcal{A}_0} V_H^{(1)} 
= \bigoplus_{H \in \mathcal{A}_0} V_H^{(2)}  
$$
with the property that the action of $N_0$ on $V_{H_0}^{(i)}$ is (isomorphic to) $M_i$ for $i \in \{1,2 \}$. But since
$V_{H_0}$ is the image of $p_{H_0}$ by Lemma \ref{lem:p0tildeM} it follows that $(M_1)_{\kk} \simeq (M_2)_{\kk}$ as $\kk N_0$-modules. But this implies that $M_1 \simeq M_2$ as $\bkappa N_0$-modules and this proves the second part of the statement.
\end{proof}

\begin{remark} For the second part of the argument, we are not replicating the argument of \cite{CFW}, Proposition 5.4, because it looks incorrect to us. Indeed, it is claimed there that the restriction of $\tilde{M}$ to $W(C_{\mathcal{B}})$,
which is equivalent to the restriction to our $N_0$,
is isomorphic to $|\mathcal{R}|M$, while it is actually
isomorphic to $\Res_{N_0} \Ind_{N_0}^W M$. Actually, already for $D_4$ one can find non-isomorphic modules $M_1,M_2$ for which the restriction of $\tilde{M}_1$ 
and $\tilde{M}_2$ to $W$ are isomorphic.
\end{remark}

\begin{proposition}
Every irreducible representation of $Br_1(W)$ has the form $\tilde{M}$ for some irreducible $M$.
\end{proposition}
\begin{proof}

Let $Q$ be an irreducible representation of $Br_1(W)$ not factoring through $\kk W$. This means that there exists some $H_0 \in \mathcal{A}$ for which $e_{H_0}$ acts nontrivially. Denote $e_0 = e_{H_0}$, and $\mathcal{A}_0$ the orbit of $H_0$ under $W$.
Since $e_0^2 = \delta e_0$ and $e_0$ acts non trivially, there exists $x_0 \in Q$ such that $e_0.x_0 = \delta x_0$. Let $W_0$ denote the subgroup fixing $H_0$ and $N_0$ its normalizer. For $s \in W_0$ we have $s.x_0 = (1/\delta) s e_0. x_0
 = (1/\delta) e_0.x_0 = (\delta/\delta) x_0 = x_0$ hence $W_0$ acts trivially on $x_0$. 
Now we set $V = \kk N_0. x_0 \subset Q$.
Note that $V$ is a $\kk N_0$-module factoring through $\kk N_0/W_0$. Moreover,
for all $w \in N_0$, we have $e_0w.x_0 = w e_0 x_0 = \delta w .x_0$ hence $e_0$ acts as $\delta. \mathrm{Id}$ on $V$.

Let $H \in \mathcal{A}_0$, and $w \in W$ such that
$H = w(H_0)$. Then $w.V$ depends only on $H$. Indeed, if $w_1^{-1}w_2 \in N_0$ we have
$w_1^{-1}w_2 V = V$ hence $w_2 V = w_1 V$. We set $V_H = w.V$. 
For $g \in W$ we have $g.V_H = gw.V = V_{H'}$
for $H' = gw(H)$.
Clearly $\sum_H V_H$ is a $\kk W$-submodule of $Q$.
Then, for $x \in V_H$ and $H = w(H_0)$, we have $e_H.x = w e_0 w^{-1}. x =w e_0.( w^{-1}. x) =
\delta w .( w^{-1}. x) = \delta. x$.
But this implies that, either $e_0.x = (1/\delta) e_0 e_H.x = 0$ or
$$
e_0.x = (1/\delta) e_0 e_H.x = 
(1/\delta)\sum_{u\in \mathcal{R}; u(H)=H_0} \mu_u u e_H.x= 
\sum_{u\in \mathcal{R}; u(H)=H_0} \mu_u u .x
$$
hence $e_0$ maps $x$ inside $V$.
In particular $\sum V_H$ is stable under $e_0$. 
Finally, if $H' \in \mathcal{A}\setminus \mathcal{A}_0$, we have $e_{H'}.x = (1/\delta) e_{H'} e_0.x = 0$ for all $x \in V$. This implies similarly that $e_{H'}.wx = w.e_{w^{-1}(H')}.x = 0$
for all $x \in V$, hence $e_{H'}$ acts by $0$
on $\sum_{H} V_H$. This implies that $\sum_H V_H$
is stable under $Br_1(W)$, hence $\sum_H V_H = Q$.

Let now assume that $V$ is not irreducible as a $\kk N_0$-module, and contains a proper irreducible submodule $U$. Then one proves similarly that $\sum_H U_H$ is stable under $Br_1(W)$, where $U_H = w.U$ for $w(H_0)=H$,
hence $\sum_H U_H = \sum_H V_H = Q$. Finally, remark that, if $x \in \sum_H V_H$ satisfies $e_0.x = \delta.x$, then $x\in V$, since we know
that $e_0.(\sum_H V_H) \subset  V$. Similarly, we check that $e_0$ maps $\sum_H U_H$ to $U$, hence
$U$ and $V$ can be both identified with the nullspace of $e_0 - \delta$ on the same space, hence $U = V$, which proves the irreducibility of $V$ as a $\kk N_0$-module.

Finally, for $H = w(H_0)$ we have
that $e_H = w e_0 w^{-1}$ maps $Q$ to $V_H = w.V$
and acts by $\delta$ on $V_H$. It follows that
$V_H$ is the nullspace of $e_H - \delta$ and this proves that $Q = \bigoplus_H V_H$. This proves that, as a $\kk W$-module, $Q$ is the induced representation of $V$. But, by Lemma \ref{lem:descent}, for any $\bkappa(\delta) N_0$-module $V$ there exists a
$\bkappa N_0$-module $V_0$ such that 
$V \simeq V_0 \otimes_{\bkappa} \kk$, hence $Q \simeq \tilde{V}_0$, and this proves the claim.

\end{proof}

\begin{remark} The determination of the representations of $Br_1(W)$ for type $H_3$ was essentially done in \cite{CHENBR}.
\end{remark}

\section{Proof of Theorem \ref{thm:structBr1}}

From Proposition \ref{prop:spannBr} we know that
$Br_1(W)$ admits a spanning set formed by the $w \in W$
and the $w e_H, H \in \mathcal{H}$, where $w \in W/W_H$
and $W_H$ is the parabolic subgroup fixing $H$. Therefore
the dimension of $Br_1(W)$ is at most $|W|+\sum_{\mathcal{A}_0 \in \mathcal{A}/W} |\mathcal{A}_0|\times |W|/|W_{H_0}|$ where $W_{H_0}$ is the pointwise stabilizer of a representative $H_0 \in \mathcal{A}_0$.

On the other hand, from the previous section we know
that $Br_1(W)$ has irreducible representations of dimension $|\mathcal{A}_0|\times \dim \theta$ where $\theta$ is an irreducible representation of $N_W(W_{H_0})/
W_{H_0}$. Therefore its dimension is at least
$$
\sum_{\mathcal{A}_0 \in \mathcal{A}/W} |\mathcal{A}_0|^2\times\sum_{\theta \in \Irr
N_W(W_{H_0})/
W_{H_0}} (\dim \theta)^2
=\sum_{\mathcal{A}_0 \in \mathcal{A}/W} |\mathcal{A}_0|^2 \frac{|N_W(W_{H_0})|}{
|W_{H_0}|}$$
But since $W/N_W(W_{H_0})$ is in bijection with $|\mathcal{A}_0|$, this is equal to the previous quantity. Therefore we get the formula for the dimension given by the theorem, and semisimplicity as well.

\section{Examples}

\subsection{Example : $G(4,2,2)$}
\label{sect:G422}
We consider the group $W = G(4,2,2)$ of rank 2, made of monomial matrices with entries in $\mu_4 = \{1,-1,\ii,-\ii\}$ and whose product of the nonzero entries belongs
to $\mu_2 = \{ -1,1 \}$. It admits the presentation
$$
 W = \langle s,t,u \ \mid \ stu=tus=ust ,\ s^2=t^2 = u^2 = 1 \rangle 
$$
and therefore an automorphism of order 3 mapping $s \mapsto t \mapsto  u \mapsto s$. It has order $16$, and
6 reflections, $s,t,u,s',t',u'$, forming 3 conjugacy classes of two elements, $\{s,s'\},\{t,t'\},\{u,u'\}$. The reflection $s'$ is
equal to $tst = t\underline{stu}u = t\underline{tus}u = usu$, hence $t'=utu=sts$, $u'= sus = tut$.

The normaliser $N_0$ of $W_0 = \langle s \rangle$ is abelian of order $8$, isomorphic to $\Z/2 \times \Z/4$.
The element $z = stu$ generates the center of $W$, has order $4$, and its image generates $N_0/W_0$.

Therefore $N_0/W_0$ is naturally identified with $Z(W)$.

The irreducible representations of $N_0$ satisfying the
property that $W_0$ acts trivially are therefore uniquely
determined by a choice of $\zeta \in \mu_4$ and given
by the formula $R_{\zeta} : z \mapsto \zeta$.
We now consider the induced module $V_s(\zeta) = \mathrm{Ind}_{N_0}^W R_{\zeta}$. It admits for basis
$v_s, v_{s'} = t.v_s$, where $N_0$ acts on $\C v_s$
by $R_{\zeta}$. 
From this we get $u.v_s = ts.stu.v_s = \zeta ts.v_s = \zeta t.v_s = \zeta v_{s'}$, hence $u.v_{s'} = u^{-1}.v_{s'}=\zeta^{-1}.v_s$, and similarly
$s.v_s = v_s$, $s.v_{s'} = u.ust.v_s = \zeta.u.v_s = \zeta^2.v_s$, and $t.v_s = v_{s'}$, $t.v_{s'}=v_s$. Therefore we get the matrices
$$
s \mapsto \begin{pmatrix}
1 & 0 \\ 0 & \zeta^2
\end{pmatrix}
t \mapsto \begin{pmatrix}
0 & 1 \\ 1 & 0
\end{pmatrix}
u \mapsto \begin{pmatrix}
0 & \zeta^{-1} \\ \zeta & 0
\end{pmatrix}
s' \mapsto \begin{pmatrix}
\zeta^2 & 0 \\ 0 & 1
\end{pmatrix}
t' \mapsto \begin{pmatrix}
0 & \zeta^2 \\ \zeta^2 & 0
\end{pmatrix}
u' \mapsto \begin{pmatrix}
0 & \zeta \\ \zeta^{-1} & 0
\end{pmatrix}
$$ 
Notice that, for $\zeta \in \{-\ii,\ii\}$, we get reflection representations for $W$. We have $s,s' \in N_0$, while $xsx = s'$ for $x \in
\{ t,u,t',u' \}$.
Letting $s_0 = s$, we get $p_s.v_s = \delta v_s$
and 
$$
p_s.v_{s'} = \sum_{xsx = s'} x.v_{s'}
= (t+u+t'+u').v_{s'} = (1+\zeta+\zeta^2+\zeta^3).v_{s'}
$$
Since $s'=tst$ we get
$$
p_s = \begin{pmatrix}
\delta & 0 \\ 0 & 1+\zeta+\zeta^2+\zeta^3
\end{pmatrix}, \ \ \ 
p_{s'} = \begin{pmatrix}
1+\zeta+\zeta^2+\zeta^3 & 0 \\ 0 & \delta
\end{pmatrix}
$$
and $p_x = 0$ for $x \in \{t,t',u,u'\}$. Note that
$1+\zeta+\zeta^2+\zeta^3 = 0$ unless $\zeta = 1$.

Condition (1)'' is always fulfilled on such representations when $p_H= 0$, so we consider only the
case where $H$ is the reflecting hyperplane of $s$ or $s'$.
But in this case the reflection under consideration has
to be the other one, and therefore we need to check
whether $s p_{s'}=p_{s'} s= p_{s'}$ and 
$s' p_{s}=p_{s} s'= p_s$. This is the case if and only if $\zeta^2 = 1$, that is $\zeta\in \{-1,1\}$.

In order that condition (1)' is fulfilled, we need first of all that (1)'' is fulfilled, so we consider only the case $\zeta^2 = 1$. Let $w \in W$ satisfying $\, ^ws = s$,
that is $w \in N_0$. Since $Z(W)$ is a complement to
$W_0$ inside $W$, and since it is generated by $z$, the
condition is then whether $z p_s = p_s z = p_s$,
which would imply $z p_{s'} = p_{s'} z = p_{s'}$ after
conjugation by $t$. But $z p_s = p_s z = \zeta p_s$,
hence the condition is fulfilled
only if $\zeta = 1$.

These two facts prove that conditions (1)' and (1)'' are
not equivalent, and that they are genuine additional conditions.

The other representations $V_{t}(\zeta)$,
$V_{u}(\zeta)$ are deduced from $V_s(\zeta)$
by applying $\varphi$, as this is readily extended
to an automorphism of order $3$ of $Br(M)$.

\subsection{Example : $G(e,e,3)$}
\label{subsect:Gee3}

For $W = G(e,e,3)$, there is a single conjugacy class of
reflections. Let us choose $s_0= \begin{pmatrix}
0 & 1 & 0 \\ 1 & 0 & 0 \\ 0& 0& 1
\end{pmatrix}$. Then $N_0$ is easily determined to be
equal the image of $\langle \mu_e, (1,2) \rangle < \GL_2(\C)$ under the map $M \mapsto \diag(M,q(M)^{-1})$
where $q(M)$ is the product of the entries of $M$. Therefore $N_0/W_0$ is isomorphic to the group of all complex $e$-th roots of $1$. Now, $W$ has order $e^2\times 3! = 6e^2$, $N_0$ has order $2e$, $\mathcal{R}$ has cardinality $3e$, hence $Br(M) = Br_1(M)$ has dimension $6e^2+9e^3 = 3e^2(2+3e)$.

\end{document}